\documentclass[11pt]{article}

\usepackage{latexsym}
\usepackage{amsmath}
\usepackage{amsfonts}
\usepackage{amssymb}
\usepackage{amsthm}
\usepackage{psfrag}
\usepackage{epsfig}
\usepackage{color}

\usepackage[tight]{subfigure}
\newcommand{\R}{\mbox{\rm I\kern-.18em R}}
\newcommand{\fR}{\mbox{\footnotesize\rm I\kern-.18em R}}
\newcommand{\sR}{\mbox{\small\rm I\kern-.18em R}}
\newcommand{\N}{\mbox{\rm I\kern-.18em N}}
\newcommand{\dist}{\mathop{\rm dist}\nolimits}

\newcommand{\CO}{{\cal O}}

\newcommand{\CPFa}{C_{\rm PF,SS}}
\newcommand{\CPFb}{C_{\rm PF,I}}

\newcommand{\be}{\begin{equation}}
\newcommand{\ee}{\end{equation}}

\newtheorem{theorem}{Theorem}[section]

\newtheorem{remark}[theorem]{Remark}
\newtheorem{lemma}[theorem]{Lemma}
\newtheorem{corollary}[theorem]{Corollary}
\newtheorem{proposition}[theorem]{Proposition}

\title{
On the equivalence of fractional-order Sobolev semi-norms
\thanks{Supported by CONICYT-Chile through Fondecyt project 1110324
        and Anillo ACT1118 (ANANUM).}
}

\author{
Norbert Heuer
\thanks{
Facultad de Matem\'aticas, Pontificia Universidad Cat\'olica de Chile,
Avenida Vicu\~na Mackenna 4860, Macul, Santiago, Chile,
email: {\tt nheuer@mat.puc.cl}}
}

\begin{document}
%\date{\today}
\date{}
\maketitle

\bigskip
\begin{abstract}
We present various results on the equivalence and mapping properties
under affine transformations
of fractional-order Sobolev norms and semi-norms of orders between zero and one.
Main results are mutual estimates of the three semi-norms of Sobolev-Slobodeckij,
interpolation and quotient space types. In particular, we show that the former two are
uniformly equivalent under affine mappings that ensure shape regularity of the
domains under consideration.


%65N55,          %Multigrid methods; domain decomposition
%65N30,          %Finite elements, Rayleigh-Ritz and Galerkin methods, finite methods
%65N12,          %Stability and convergence of numerical methods
%65N15.          %Error bounds
\end{abstract}

%%%%%%%%%%%%%%%%%%%%%%%%%%%%%%%%%%%%%%%%%%%%%%%%%%%%%%%%%%%%%%%%%%%%%%%%%%%%%%%%
\section{Introduction}

Sobolev norms and semi-norms play a central role in the numerical analysis
of discretization methods for partial differential equations. For instance,
standard finite element error analysis is essentially a combination of
the Bramble-Hilbert lemma and transformation properties of Sobolev (semi-) norms.
These properties are also central to the area of preconditioners for (and
based on) variational methods. More precisely, arguments based on finite dimensions
of local spaces are inherently connected with scaling arguments to keep dimensions
bounded. Norms are usually not scalable. That is, the corresponding equivalence numbers
behave differently with respect to a scaling parameter like the diameter $D_{\CO}$
of the domain $\CO$ when the domain under consideration is transformed by an affine map
that maintains shape regularity (i.e., the ratio of $D_{\CO}$ and the ``inner diameter''
of $\CO$ is bounded). This can be usually fixed only when
essential boundary conditions are present. An example is using the $H^1$-semi-norm
as norm in $H^1_0$. More generally, semi-norms have better scaling properties: usually
they can be defined so that equivalence numbers are of the same order with
respect to $D_{\CO}$ under shape-regular affine transformations of the domain.

Whereas properties of Sobolev (semi-) norms under smooth transformations or
simple scalings are straightforward as long as their orders are integer, things
are getting more complicated for fractional-order Sobolev norms. Such norms appear,
e.g., in a natural way when considering boundary integral equations of the
first kind \cite{McLean_00_SES,HsiaoW_08_BIE}
or when studying the regularity of elliptic problems in non-convex
polygonal domains \cite{Grisvard_85_EPN}.
There are different
ways to define fractional-order Sobolev norms and they all have advantages and
disadvantages (standard references are \cite{LionsMagenes,Adams}).
Different norm variants are known to be equivalent. But dependence
of the equivalence constants on the order and the domain are more involved.

In this paper we analyze the equivalence of different variants of fractional-order
semi-norms of positive orders bounded by one. The use of semi-norms is essential
to guarantee scaling properties and we don't know of any publication that analyzes
their equivalence.
%Instead of considering general smooth transformations that do
%not distort domains, we focus on pure isotropic scalings of domains. (More general
%transformations can be generated by composition with translations and area preserving
%smooth mappings.) In this way we keep appearing expressions of equivalence numbers
%as transparent as possible.

The rest of the paper is organized as follows. In Section~\ref{sec_S} we
collect all definitions and technical results.
In Section~\ref{sec_S1} we recall two definitions of norms and
define three different semi-norms: one of the Sobolev-Slobodeckij type, one by
interpolation, and one of a quotient space type. Section~\ref{sec_S3} is devoted
to basic equivalence estimates. In particular, we present Poincar\'e-Friedrichs'
inequalities for the Sobolev-Slobodeckij and interpolation semi-norms
(Propositions~\ref{prop_PF1} and \ref{prop_PF2}).
A direct proof in the case of the Sobolev-Slobodeckij semi-norm
is cited from Faermann \cite{Faermann_02_LAS}. An indirect proof for
the interpolation semi-norm is standard, and is given for completeness.
Affine transformation properties of norms
(also given for completeness) and semi-norms are analyzed in Section~\ref{sec_scaling}.
Eventually, in Section~\ref{sec_main} we combine the intermediate results
to show the equivalence of all the semi-norms under consideration, with
explicit equivalence numbers depending on the domain via its transformation from
a reference domain (Theorems~\ref{thm1}--\ref{thm3}).
In Theorem~\ref{thm_main} we resume the results in a form that is appropriate
for affine maps that maintain shape-regularity of the domain.
In particular, it shows (i) the uniform (with respect to $D_{\CO}$)
equivalence of the Sobolev-Slobodeckij and the interpolation semi-norms,
(ii) that the Sobolev-Slobodeckij and quotient space semi-norms are uniformly
equivalent as long as the diameter of the domain is bounded from
above, and (iii) that the interpolation and quotient space semi-norms
can be uniformly bounded mutually in one direction depending on whether the diameter
of the domain is bounded from above or from below.
Finally, in Corollary~\ref{cor_scale} we collect the scaling properties
of all the norms and semi-norms studied in this paper that have this property.

%%%%%%%%%%%%%%%%%%%%%%%%%%%%%%%%%%%%%%%%%%%%%%%%%%%%%%%%%%%%%%%%%%%%%%%%%%%%%%%%
\section{Sobolev norms} \label{sec_S}

In this section we recall definitions of several Sobolev (semi-) norms and collect
technical results that are needed to prove our main results in Section~\ref{sec_main},
or which are interesting in its own.

Throughout the paper, $\CO\subset\R^n$ denotes a generic
bounded connected Lipschitz domain.
We consider the usual $L^2(\CO)$- and
$H^1(\CO)$-norms with notations $\|\cdot\|_{0,\CO}$ and
$\|\cdot\|_{1,\CO}$, respectively and the $H^1(\CO)$-semi-norm
$|\cdot|_{1,\CO}$. Here and in the following, in all types of norms,
the underlying domain of definition $\CO$ will be occasionally dropped from the
notation when not being ambiguous.

%%%%%%%%%%%%%%%%%%%%%%%%%%%%%%%%%%%%%%%%%%%%%%%%%%%%%%%%%%%%%%%%%%%%%%%%%%%%%%%%
\subsection{Fractional-order norms and semi-norms} \label{sec_S1}

There are several ways to define Sobolev norms. We use the ones defined by
a double integral (Sobolev-Slobodeckij) and by interpolation. For the latter we
use the so-called real K-method, cf.~\cite{BerghL_76_IS}.
For $0<s<1$, the interpolation norm in the fractional-order Sobolev space $H^s(\CO)$
is defined by
\[
   \|v\|_{[L^2(\CO),H^1(\CO)]_s}
   :=
   \|v\|_{L^2(\CO),H^1(\CO),s}
   :=
   \left( \int_0^\infty
          t^{-2s} \inf_{v=v_0+v_1}
          \Bigl(\|v_0\|_{0,\CO}^2 + t^2\, \|v_1\|_{1,\CO}^2\Bigr)
          \frac{dt}{t}
   \right)^{1/2}.
\]
Here and in the following, the notation
$\inf_{v=v_0+v_1} \Bigl(\|v_0\|_{0,\CO}^2 + t^2\, \|v_1\|_{1,\CO}^2\Bigr)$
implies that the infimum is taken over $v_0\in L^2(\CO)$ and $v_1\in H^1(\CO)$,
or corresponding spaces as indicated by the respective norms.

We also define the interpolation space
\[
   \tilde H^s(\CO) = \left[ L^2(\CO), H^1_0(\CO) \right]_s
\]
with corresponding notation for the norm.
The notation $\tilde H^s$ is used by Grisvard and is common in the
boundary element literature, whereas the notation
$H_{00}^s=\tilde H^s$
is used by Lions and Magenes and is common in the finite element literature.

The Sobolev-Slobodeckij variant of these norms is defined (for $0<s<1$) by
\begin{equation} \label{SS-norm}
  \|v\|_{H^s(\CO)}
  :=
  \|v\|_{s,\CO}
  :=
  \left(\|v\|_{L^2(\CO)}^2
  +
  \int_{\CO} \int_{\CO} \frac{|v(x)-v(y)|^2}{|x-y|^{n+2s}}\,dx\,dy
  \right)^{1/2},
\end{equation}
\[
  \|v\|_{\tilde H^s(\CO)}
  :=
  \|v\|_{\sim,s,\CO}
  :=
  \left(
  \|v\|_{H^s(\CO)}^2
  +
  \|\frac{v(x)}{\dist(x,\partial \CO)^s}\|_{L^2(\CO)}^2
  \right)^{1/2}
  \qquad\text{(preliminary version)}.
\]
The corresponding semi-norms are
\[
   |v|_{[L^2(\CO),H^1(\CO)]_s}
   :=
   |v|_{L^2(\CO),H^1(\CO),s}
   :=
   \left( \int_0^\infty t^{-2s}
             \inf_{v=v_0+v_1} \Bigl(\|v_0\|_{0,\CO}^2 + t^2\, |v_1|_{1,\CO}^2\Bigr)
          \frac{dt}{t}
   \right)^{1/2}
\]
and
\[
  |v|_{H^s(\CO)}
  :=
  |v|_{s,\CO}
  :=
  \left(
  \int_{\CO} \int_{\CO} \frac{|v(x)-v(y)|^2}{|x-y|^{n+2s}}\,dx\,dy
  \right)^{1/2}.
\]
Additionally, it is useful to define the semi-norm of quotient space type
\[
  |v|_{s,\CO,\inf} := \|v\|_{H^s(\CO)/\fR}
                      = \inf_{c\in\fR} \|v+c\|_{s,\CO}.
\]

%%%%%%%%%%%%%%%%%%%%%%%%%%%%%%%%%%%%%%%%%%%%%%%%%%%%%%%%%%%%%%%%%%%%%%%%%%%%%%%%
\subsection{Equivalence of semi-norms on a fixed domain} \label{sec_S3}

The aim of this section is to study equivalences of the semi-norms previously defined,
on a fixed domain. Together with mapping properties (provided in Section~\ref{sec_scaling})
these estimates are needed to prove our main results in Section~\ref{sec_main}.
Proofs are based on a standard norm equivalence and specific Poincar\'e-Friedrichs'
inequalities, which are also recalled here.

It is well known that for Lipschitz domains different definitions of Sobolev norms are
equivalent. However, equivalence constants depend usually on the order and the
domain under consideration. In particular, for a bounded
Lipschitz domain $\CO$, the norms $\|\cdot\|_{s,\CO}$ and
$\|\cdot\|_{L^2(\CO),H^1(\CO),s}$ are equivalent for $0<s<1$,
cf.~\cite{LionsMagenes,Grisvard_85_EPN,McLean_00_SES}.
Such equivalences are shown by corresponding equivalences
on $\R^n$ and the use of appropriate extension operators,
cf.~\cite{BrennerS_94_MTF}, see also \cite{DevoreS_93_BSD} for non-Lipschitz
domains. In particular, the norms previously defined are uniformly equivalent
for $s$ in a closed subset of $(0,1)$, see~\cite{Heuer_01_ApS}.

Here, for the norms, we don't elaborate on the dependence of the equivalence
constants on $s$ and $\CO$.
We rather give them specific names to be used in estimates to follow.

\begin{proposition}[equivalence of norms] \label{prop_norms}
For a bounded Lipschitz domain $\CO\subset\R^n$ and for given
$s\in(0,1)$ there exist constants $k(s,\CO), K(s,\CO)>0$ such that
\[
   k(s,\CO)\,\|v\|_{L^2(\CO),H^1(\CO),s}
   \le \|v\|_{s,\CO}
   \le K(s,\CO)\, \|v\|_{L^2(\CO),H^1(\CO),s}
   \qquad\forall v\in H^s(\CO).
\]
\end{proposition}

For a proof see, e.g., \cite{McLean_00_SES}.

It is well known that, on bounded Lipschitz domains, lower-order norms can be bounded
by higher-order semi-norms plus finite rank terms. Such estimates are referred to
as Poincar\'e-Friedrichs' inequalities. For integer-order norms
there are direct proofs with explicit constants (depending on orders and
domains) \cite[Th\'eor\`eme 1.3]{Necas_67_MDT} and attention has received
finding best constants and deriving improved weighted estimates, see, e.g.,
\cite{PayneW_60_OPI,Verfuerth_99_NPA} and \cite{DrelichmanD_08_IPI}, respectively.
We need such a Poincar\'e-Friedrichs' inequality for fractional-order
norms on bounded domains (for unbounded domains, see \cite{MouhotRS_11_FPI}),
and refer to \cite[Lemma 3.4]{Faermann_02_LAS} for a proof.
This proof is given for two dimensions but immediately extends to the general case.

\begin{proposition}[Poincar\'e-Friedrichs inequality, Sobolev-Slobodeckij semi-norm]
\label{prop_PF1}
Let $\CO\subset\R^n$ be a bounded domain, and $s\in(0,1)$. Then there holds
\[
   \|v\|_{0,\CO} \le \CPFa(s,\CO) \Bigl(|v|_{s,\CO} + |\int_\CO v|\Bigr)
   \qquad\forall v\in H^s(\CO)
\]
with
\[
   \CPFa(s,\CO) = |\CO|^{-1/2} \max\{1,2^{-1/2} D_{\CO}^{n/2+s}\}.
\]
Here, $|\CO|$ denotes the area of $\CO$ and, as mentioned in the introduction,
$D_{\CO}$ is its diameter.
\end{proposition}

\begin{lemma} \label{la1}
Let $\CO\subset\R^n$ be a bounded, connected Lipschitz domain.
Then there holds
\[
   |v|_{s,\CO}^2
   \le |v|_{s,\CO,\inf}^2
   =
   |v|_{s,\CO}^2 + \inf_{c\in\fR} \|v+c\|_{0,\CO}^2
   \le
   (1+\CPFa^2) |v|_{s,\CO}^2
\]
for any $v\in H^s(\CO)$ and $s\in (0,1)$.
Here, $\CPFa=\CPFa(s,\CO)$ is the number from Proposition~\ref{prop_PF1}.
\end{lemma}

\begin{proof}
By definition of $|\cdot|_{s,\CO}$ there holds for any $c\in\R$
and any $v\in H^s(\CO)$ (we now drop $\CO$ from the notation)
\[
   |v|_{s} = |v+c|_{s}.
\]
Therefore
\[
   |v|_{s} \le \inf_{c\in\fR} \|v+c\|_{s} = |v|_{s,\inf}
\]
which is the first assertion.
By the initial argument and the definition of the Sobolev-Slobodeckij norm
one also finds that
\[
   |v|_{s,\inf}^2
   =
   \inf_{c\in\fR} \|v+c\|_{s}^2
   = \inf_{c\in\fR} \|v+c\|_{0}^2 + |v|_{s}^2.
\]
This is the second assertion.

The last relation and the Poincar\'e-Friedrichs' inequality (Proposition~\ref{prop_PF1})
lead to
\[
   |v|_{s,\inf}^2
   \le
   \CPFa^2 \inf_{c\in\fR} \Bigl( |v|_s + |\int_\CO (v+c)|\Bigr)^2 + |v|_{s}^2
   =
   (1+\CPFa^2) |v|_{s}^2.
\]
This finishes the proof.
\end{proof}

\begin{lemma} \label{la2}
Let $\CO\subset\R^n$ be a bounded Lipschitz domain. There holds
\[
   k^2
   |v|_{L^2(\CO),H^1(\CO),s}^2
   \le
   |v|_{s,\CO,\inf}^2
   \le
   3K^2 |v|_{L^2(\CO),H^1(\CO),s}^2
   +
   \frac {K^2}{s(1-s)}
   \inf_{c\in\fR} \|v+c\|_{0,\CO}^2
\]
for any $v\in H^s(\CO)$ and $s\in (0,1)$.
Here, $k=k(s,\CO)$ and $K=K(s,\CO)$ are the numbers from Proposition~\ref{prop_norms}.
\end{lemma}

\begin{proof}
Let $v\in H^s(\CO)$, and let $c_0, c_1$ denote generic constants.
For any $t>0$ there holds
\begin{align*}
   \inf_{v=v_0+v_1} \Bigl(\|v_0\|_{0}^2 + t^2 |v_1|_{1}^2\Bigr)
   &=
   \inf_{v=v_0+c_0+v_1+c_1} \Bigl(\|v_0+c_0\|_{0}^2 + t^2 |v_1|_{1}^2\Bigr)\\
   &=
   \inf_{c_1, v-c_1=v_0+v_1} \Bigl(\|v_0\|_{0}^2 + t^2 |v_1|_{1}^2\Bigr),
\end{align*}
that is
\begin{align*}
   \inf_{v=v_0+v_1} \Bigl(\|v_0\|_{0}^2 + t^2 |v_1|_{1}^2\Bigr)
   &=
   \inf_{c\in\fR} \inf_{v+c=v_0+v_1} \Bigl(\|v_0\|_{0}^2 + t^2 |v_1|_{1}^2\Bigr)\\
   &\le
   \inf_{c\in\fR} \inf_{v+c=v_0+v_1} \Bigl(\|v_0\|_{0}^2 + t^2 \|v_1\|_{1}^2\Bigr).
\end{align*}
(Recall that our convention for the notation
$\inf_{v=v_0+v_1} \Bigl(\|v_0\|_{0}^2 + t^2 |v_1|_{1}^2\Bigr)$
implies that the infimum is taken with respect to $v_0\in L^2(\CO)$ and $v_1\in H^1(\CO)$.)
We conclude that
\begin{align*}
   |v|_{L^2,H^1,s}^2
   &=
   \int_0^\infty t^{-2s}
       \inf_{v=v_0+v_1} \Bigl(\|v_0\|_{0}^2 + t^2\, |v_1|_{1}^2\Bigr)
   \frac{dt}{t}
   \\
   &\le
   \inf_{c\in\fR}
   \int_0^\infty t^{-2s}
       \inf_{v+c=v_0+v_1} \Bigl(\|v_0\|_{0}^2 + t^2\, \|v_1\|_{1}^2\Bigr)
   \frac{dt}{t}
   =
   \inf_{c\in\fR} \|v+c\|_{L^2,H^1,s}^2.
\end{align*}
By Proposition~\ref{prop_norms}
\[
   \inf_{c\in\fR} \|v+c\|_{L^2,H^1,s}^2
   \le
   k^{-2} \inf_{c\in\fR} \|v+c\|_{s}^2
   =
   k^{-2} |v|_{s,\inf}^2,
\]
so that the first assertion follows.

By definition and using Proposition~\ref{prop_norms} there holds
\begin{align} \label{1}
   |v|_{s,\inf}^2
   &=
   \inf_{c\in\fR} \|v+c\|_{s}^2
   \le
   K^2 \inf_{c\in\fR} \|v+c\|_{L^2,H^1,s}^2
   \nonumber\\
   &=
   K^2
   \inf_{c\in\fR}
   \int_0^\infty t^{-2s}
      \inf_{v+c=v_0+v_1} \Bigl(\|v_0\|_{0}^2 + t^2 \|v_1\|_{0}^2 + t^2 |v_1|_{1}^2\Bigr)
   \frac{dt}{t}.
\end{align}
We bound the integrand separately for $t<1$ and $t\ge 1$.

For $t<1$ we use the representation $v+c=v_0+v_1$ to bound
\begin{align*}
   \|v_0\|_{0}^2 + t^2 \|v_1\|_{0}^2 + t^2 |v_1|_{1}^2
   &\le
   \|v_0\|_{0}^2 + 2 t^2\bigl( \|v+c\|_{0}^2 + \|v_0\|_{0}^2\bigr)
                     + t^2 |v_1|_{1}^2
   \\
   &\le
   3 \|v_0\|_{0}^2 + 2 t^2 \|v+c\|_{0}^2 + t^2 |v_1|_{1}^2.
\end{align*}
If $t\ge 1$ then we select $v_0:=v+c$ to conclude that
\[
   \inf_{v+c=v_0+v_1}
   \Bigl(\|v_0\|_{0}^2 + t^2 \|v_1\|_{0}^2 + t^2 |v_1|_{1}^2\Bigr)
   \le
   \|v+c\|_{0}^2.
\]
Together this yields
\begin{align} \label{pf_la2_1}
   &\int_0^\infty t^{-2s}
       \inf_{v+c=v_0+v_1}
       \Bigl(\|v_0\|_{0}^2 + t^2 \|v_1\|_{0}^2 + t^2 |v_1|_{1}^2\Bigr)
    \frac{dt}{t}
   \nonumber\\
   \le
   &\int_0^1 t^{-2s}
       \inf_{v+c=v_0+v_1}
       \Bigl( 3 \|v_0\|_{0}^2 + 2 t^2 \|v+c\|_{0}^2 + t^2 |v_1|_{1}^2 \Bigr)
    \frac{dt}{t}
   +
   \int_1^\infty t^{-2s} \|v+c\|_{0}^2 \frac{dt}{t}
   \nonumber\\
   =
   &\int_0^1
   t^{-2s} \inf_{v+c=v_0+v_1}
   \Bigl( 3 \|v_0\|_{0}^2 + t^2 |v_1|_{1}^2 \Bigr) \frac{dt}{t}
   +
   \|v+c\|_{0}^2
   \Bigl( \int_0^1 2 t^{1-2s} \,dt + \int_1^\infty t^{-1-2s} \,dt \Bigr)
   \nonumber\\
   &\le
   3 |v|_{L^2,H^1,s}^2 + \frac 1{s(1-s)} \|v+c\|_{0}^2.
\end{align}
Therefore, recalling \eqref{1}, we obtain
\[
   |v|_{s,\inf}^2
   \le
   3K^2 |v|_{L^2,H^1,s}^2 + \frac {K^2}{s(1-s)} \inf_{c\in\fR} \|v+c\|_{0}^2,
\]
which is the second assertion.
\end{proof}

From the proof of the previous lemma one can conclude that the semi-norm
$|\cdot|_{L^2(\CO),H^1(\CO),s}$ is indeed the principal part of a norm in
$H^s(\CO)$. This will be useful to deduce a Poincar\'e-Friedrichs inequality
with this semi-norm. First let us specify what we mean by the semi-norm
being principal part of a norm.

\begin{corollary} \label{cor}
Let $\CO\subset\R^n$ be a bounded Lipschitz domain. There holds
\[
   \|v\|_{s,\CO}^2
   \le
   \frac {K^2}{s(1-s)}
   \|v\|_{0,\CO}^2
   +
   3K^2 |v|_{L^2(\CO),H^1(\CO),s}^2
\]
for any $v\in H^s(\CO)$ and $s\in (0,1)$.
Here, $K=K(s,\CO)$ is the number from Proposition~\ref{prop_norms}.
\end{corollary}

\begin{proof}
This is a combination of the second bound from Proposition~\ref{prop_norms}
and \eqref{pf_la2_1} with $c=0$.
\end{proof}

We are now ready to prove a second Poincar\'e-Friedrichs inequality.

\begin{proposition}[Poincar\'e-Friedrichs inequality, interpolation semi-norm]
\label{prop_PF2}
Let $\CO\subset\R^n$ be a bounded connected Lipschitz domain, and $s\in(0,1)$. Then there
exists a constant $\CPFb>0$, depending on $\CO$ and $s$, such that
\[
   \|v\|_{0,\CO} \le \CPFb(s,\CO) \Bigl(|v|_{L^2(\CO), H^1(\CO), s} + |\int_\CO v|\Bigr)
   \qquad\forall v\in H^s(\CO).
\]
\end{proposition}

\begin{proof}
Assume that the inequality is not true. Then there is a sequence
$(v_j)\subset H^s(\CO)$ such that
\[
   \|v_j\|_{0,\CO}=1,\quad |v_j|_{L^2(\CO), H^1(\CO), s} + |\int_\CO v_j|\to 0\ (j\to\infty).
\]
Therefore, by  Corollary~\ref{cor}, $(v_j)$ is bounded in $H^s(\CO)$
with respect to the Sobolev-Slobodeckij norm.
Then, by Rellich's theorem (see \cite[Theorem 3.27]{McLean_00_SES})
there is a convergent subsequence (again denoted by $(v_j)$) in $L^2(\CO)$.
Since $|v_j|_{L^2(\CO), H^1(\CO), s}\to 0$ this sequence is Cauchy and with limit $v$ in $H^s(\CO)$.
It holds $|v|_{L^2(\CO), H^1(\CO), s}=0$ so that $v$ is constant.
Furthermore, since $\int_\CO v=0$ and $\CO$ is connected we conclude that $v=0$,
a contradiction to $\|v_j\|_{0,\CO}=1$.
\end{proof}

With the help of Proposition~\ref{prop_PF2} we can now turn the estimate by
Lemma~\ref{la2} into a semi-norm equivalence.

\begin{lemma} \label{la3}
Let $\CO\subset\R^n$ be a connected bounded Lipschitz domain. There holds
\[
   k^2
   |v|_{L^2(\CO),H^1(\CO),s}^2
   \le
   |v|_{s,\CO,\inf}^2
   \le
   K^2
   \Bigl( 3 + \frac {\CPFb^2}{s(1-s)} \Bigr)
   |v|_{L^2(\CO),H^1(\CO),s}^2
\]
for any $v\in H^s(\CO)$ and $s\in (0,1)$.
Here, $k=k(s,\CO)$, $K=K(s,\CO)$ are the numbers from Proposition~\ref{prop_norms},
and $\CPFb=\CPFb(s,\CO)$ is the number from Proposition~\ref{prop_PF2}.
\end{lemma}

\begin{proof}
The lower bound is the one from Lemma~\ref{la2}. The upper bound is a combination
of the upper bound from the same lemma and the Poincar\'e-Friedrichs' inequality from
Proposition~\ref{prop_PF2}. To this end note that the infimum
$\inf_{c\in\fR}\|v+c\|_{0,\CO}$ is achieved by the same constant $c$ that
eliminates the integral in the bound of the Poincar\'e-Friedrichs' inequality for $v+c$.
\end{proof}

Meanwhile we have accumulated quite some parameters in the semi-norm estimates
that depend on the order $s$ and the domain $\CO$ under consideration.
Our goal is to show equivalence of semi-norms which is uniform for a family of
affinely transformed domains. We therefore study transformation properties of
semi-norms in the following section. In this way, parameters from this section enter
final results only via their values on a reference domain.

%%%%%%%%%%%%%%%%%%%%%%%%%%%%%%%%%%%%%%%%%%%%%%%%%%%%%%%%%%%%%%%%%%%%%%%%%%%%%%%%
\subsection{Transformation properties of norms and semi-norms} \label{sec_scaling}

Obviously, both norms in $H^s(\CO)$ defined previously,
$\|\cdot\|_{L^2(\CO),H^1(\CO),s}$ and $\|\cdot\|_{s,\CO}$, are not scalable.
This could be achieved by weighting the $L^2(\CO)$-contributions according
to the diameter of $\CO$, for instance, cf. \cite{ErvinH_06_ABS}.
Of course, in this way one does not obtain uniformly equivalent norms
(of un-weighted and weighted variants) under transformation of the domain.

This is different for the norm in $\tilde H^s(\CO)$. It can be easily
fixed (to be scalable) by using that the semi-norm $|\cdot|_{1,\CO}$ is
a norm in $H^1_0(\CO)$, and re-defining
\[
   \|v\|_{[L^2(\CO),H^1_0(\CO)]_s}
   :=
   \|v\|_{L^2(\CO),H^1_0(\CO),s}
   :=
   \left( \int_0^\infty
          t^{-2s} \inf_{v=v_0+v_1, v_1\in H^1_0(\CO)}
                  \Bigl(\|v_0\|_{0,\CO}^2 + t^2\, |v_1|_{1,\CO}^2\Bigr)
          \frac{dt}{t}
   \right)^{1/2}
\]
in the case of interpolation. In the case of the Sobolev-Slobodeckij norm
one can ensure scalability by re-defining
\[
  \|v\|_{\tilde H^s(\CO)}
  :=
  \|v\|_{\sim,s,\CO}
  :=
  \left(
  |v|_{H^s(\CO)}^2
  +
  \|\frac{v(x)}{\dist(x,\partial \CO)^s}\|_{L^2(\CO)}^2
  \right)^{1/2}
\]
since the last term guarantees positivity. In the following we will make
use of these re-defined norms.

For a domain $\hat\CO\in\R^n$ we denote by $\CO=F(\hat\CO)$ the affinely transformed domain
\be \label{trafo}
   \CO:=\{F\hat x;\, \hat x\in\hat\CO\}\quad\text{with}\quad F\hat x=x_0+B\hat x,\
   x_0\in\R^n,\ B\in\R^{n\times n}.
\ee
Here, $B$ is assumed to be invertible.
Correspondingly, for a given real function $v$ defined on $\CO$,
\[
   \hat v:\left\{\begin{array}{lll}
            \hat\CO & \to & \R\\
            \hat x & \mapsto & v(F\hat x)
       \end{array}\right.
\]
is the function transformed onto $\hat\CO$.

\begin{lemma}[transformation properties of norms] \label{la_scale}
Let $\hat\CO\subset\R^n$ be a bounded Lipschitz domain and let $\CO$ be the affinely
transformed domain defined by \eqref{trafo}. Then there hold the transformation properties
\be \label{scale_int}
   |\det B|\,\|B\|^{-2s}\; \|\hat v\|_{L^2(\hat \CO),H^1_0(\hat \CO),s}^2
   \le
   \|v\|_{L^2(\CO),H^1_0(\CO),s}^2
   \le
   |\det B|\,\|B^{-1}\|^{2s}\; \|\hat v\|_{L^2(\hat \CO),H^1_0(\hat \CO),s}^2,
\ee
\be \label{scale_tilde}
\begin{split}
   |\det B|\,\|B\|^{-2s}\min\{|\det B|\,\|B\|^{-n},1\}\; \|\hat v\|_{\sim,s,\hat\CO}^2
   \le
   \|v\|_{\sim,s,\CO}^2
\qquad\qquad\qquad\qquad\qquad
\\
\qquad\qquad\qquad\qquad\qquad
   \le
   |\det B|\,\|B^{-1}\|^{2s}\max\{|\det B|\,\|B^{-1}\|^n,1\}\; \|\hat v\|_{\sim,s,\hat\CO}^2
\end{split}
\ee
for any $\hat v\in \tilde H^s(\hat\CO)$ and $s\in (0,1)$.
\end{lemma}

\begin{proof}
For the interpolation norm and $\hat\CO$, $\CO$ being a cubes, this property (with an unspecified
equivalence constant) has been shown in \cite{Heuer_01_ApS}. It is simply the scaling
properties of the $L^2$ and $H^1_0$-norms together with the exactness of the K-method
of interpolation (employed here).
The proof generalizes to affine mappings in a straightforward way as follows.
In Euclidean norm one has $\|\nabla v(x)\| \le \|B^{-1}\|\,\|\nabla \hat v(\hat x)\|$ so
that the following relations are immediate,
\[
   \|v\|_{L^2(\CO)}^2 = |\det B|\,\|\hat v\|_{L^2(\hat\CO)}^2,\quad
   |v|_{H^1(\CO)}^2 \le |\det B|\, \|B^{-1}\|^2 |\hat v|_{H^1(\hat\CO)}^2.
\]
Then, with transformation $r=\|B^{-1}\|\,t$, we deduce that
\begin{align*}
   \|v\|_{L^2(\CO),H^1_0(\CO),s}^2
   &=
   \int_0^\infty t^{-2s} \inf_{v=v_0+v_1, v_1\in H^1_0(\CO)}
          \Bigl(\|v_0\|_{0,\CO}^2 + t^2\, |v_1|_{1,\CO}^2\Bigr)
          \frac{dt}{t}
   \\
   &\le
   |\det B|
   \int_0^\infty t^{-2s} \inf_{\hat v=\hat v_0+\hat v_1, \hat v_1\in H^1_0(\hat \CO)}
          \Bigl(\|\hat v_0\|_{0,\hat \CO}^2 + t^2\, \|B^{-1}\|^2|\hat v_1|_{1,\hat \CO}^2\Bigr)
          \frac{dt}{t}
   \\
   &=
   |\det B|
   \int_0^\infty (\|B^{-1}\|^{-1}\,r)^{-2s} \inf_{\hat v=\hat v_0+\hat v_1, \hat v_1\in H^1_0(\hat \CO)}
          \Bigl(\|\hat v_0\|_{0,\hat \CO}^2 + r^2\, |\hat v_1|_{1,\hat \CO}^2\Bigr)
          \frac{dr}{r}
   \\
   &=
   |\det B|\,\|B^{-1}\|^{2s} \|\hat v\|_{L^2(\hat \CO),H^1_0(\hat \CO),s}^2.
\end{align*}
This proves the upper bound in \eqref{scale_int}. The lower bound is verified by using
the inverse transformation $F^{-1}$ with matrix $B^{-1}$.

The transformation property of the second norm is obtained similarly, see also
\cite[page 461]{DupontS_80_PAF} for the term of the double integral.
\begin{align*}
\lefteqn{
   \|v\|_{\sim,s,\CO}^2
   =
   \int_{\CO} \int_{\CO} \frac{|v(x)-v(y)|^2}{|x-y|^{n+2s}}\,dx\,dy
   +
   \int_{\CO} \Bigl(\frac{v(x)}{\dist(x,\partial \CO)^s}\Bigr)^2 \,dx
}
   \\
   &\le
   |\det B|^2
   \int_{\hat\CO} \int_{\hat\CO}
   \frac{|\hat v(\hat x)-\hat v(\hat y)|^2}{\|B^{-1}\|^{-n-2s}|\hat x-\hat y|^{n+2s}}
   \,d\hat x\,d\hat y
   +
   |\det B|
   \int_{\hat\CO} \Bigl(\frac{\hat v(\hat x)}
                             {\|B^{-1}\|^{-s}\dist(\hat x,\partial \hat\CO)^s}
                  \Bigr)^2 \,d\hat x
   \\
   &\le
   |\det B|\,\|B^{-1}\|^{2s}\max\{|\det B|\,\|B^{-1}\|^n,1\} \|\hat v\|_{\sim,s,\hat\CO}^2.
\end{align*}
This is the upper bound in \eqref{scale_tilde}. Analogously one finds that
\[
   \|\hat v\|_{\sim,s,\hat\CO}^2
   \le
   |\det B^{-1}|\,\|B\|^{2s}\max\{|\det B^{-1}|\,\|B\|^n,1\} \|v\|_{\sim,s,\CO}^2.
\]
This proves the lower bound in \eqref{scale_tilde}.
\end{proof}

\begin{lemma}[transformation properties of semi-norms] \label{la_scale_semi}
Let $\hat\CO\subset\R^n$ be a bounded Lipschitz domain and let $\CO$ be the affinely
transformed domain defined by \eqref{trafo}. Then there hold the transformation properties
\be \label{scale_semi_int}
   |\det B|\,\|B\|^{-2s}\; |\hat v|_{L^2(\hat \CO),H^1(\hat \CO),s}^2
   \le
   |v|_{L^2(\CO),H^1(\CO),s}^2
   \le
   |\det B|\,\|B^{-1}\|^{2s}\; |\hat v|_{L^2(\hat \CO),H^1(\hat \CO),s}^2,
\ee
\be \label{scale_semi_tilde}
   |\det B|^2\,\|B\|^{-n-2s}\; |\hat v|_{s,\hat\CO}^2
   \le
   |v|_{s,\CO}^2
   \le
   |\det B|^2\,\|B^{-1}\|^{n+2s}\; |\hat v|_{s,\hat\CO}^2
\ee
for any $\hat v\in H^s(\hat\CO)$ and $s\in (0,1)$.
\end{lemma}

\begin{proof}
The proof is basically identical to the one of Lemma~\ref{la_scale}.
\end{proof}

The third semi-norm, $|\cdot|_{s,\CO,\inf}$, behaves under affine transformations
as follows.

\begin{lemma} \label{la_scale_semi3}
Let $\hat\CO\subset\R^n$ be a bounded Lipschitz domain and let $\CO$ be the affinely
transformed domain defined by \eqref{trafo}. Then there hold the transformation properties
\[
\begin{split}
   |\det B|^2\,\|B\|^{-n-2s}\; |\hat v|_{s,\hat\CO}^2
   +
   |\det B|\, \inf_{c\in\fR} \|\hat v+c\|_{0,\hat\CO}^2
   \le
   |v|_{s,\CO,\inf}^2
\qquad\qquad\qquad\qquad\qquad
\\
\qquad\qquad\qquad\qquad\qquad
   \le
   |\det B|^2\,\|B^{-1}\|^{n+2s}\; |\hat v|_{s,\hat\CO}^2
   +
   |\det B|\, \inf_{c\in\fR} \|\hat v+c\|_{0,\hat\CO}^2
\end{split}
\]
for any $\hat v\in H^s(\hat\CO)$ and $s\in (0,1)$.
\end{lemma}

\begin{proof}
This result is immediate from the representation of the semi-norm
given in Lemma~\ref{la1} and the transformation properties of the
$|\cdot|_{s}$-semi-norm by Lemma~\ref{la_scale_semi} and of the $L^2$-norm.
\end{proof}

%%%%%%%%%%%%%%%%%%%%%%%%%%%%%%%%%%%%%%%%%%%%%%%%%%%%%%%%%%%%%%%%%%%%%%%%%%%%%%%%
\section{Main results} \label{sec_main}

We are now ready to state and prove our main results on certain equivalences
of fractional-order Sobolev semi-norms.
We use the notation \eqref{trafo} from Section~\ref{sec_scaling} for
affine transformations. In particular, we assume that the domain $\CO$ under
consideration is the affine image of a bounded Lipschitz domain $\hat\CO$.
The following results specify how equivalence constants depend on the affine map.
At the end of this section we conclude the equivalence of some semi-norms which
is uniform for a family of so-called shape regular domains (Theorem~\ref{thm_main})
and some scaling properties (Corollary~\ref{cor_scale}).
These results are of importance for the approximation
theory of piecewise polynomial spaces in fractional-order Sobolev spaces.

The first theorem shows the equivalence of
the semi-norms $|\cdot|_{L^2(\CO),H^1(\CO),s}$ and $|\cdot|_{s,\CO}$.

\begin{theorem} \label{thm1}
Let $\hat\CO\subset\R^n$ be a bounded, connected Lipschitz domain and let $\CO$
be the affinely transformed domain defined by \eqref{trafo}.
Then there hold the following relations.
\begin{enumerate}
\item[(i)]
\[
   |v|_{s,\CO}^2
   \le
   |\det B|\,\|B^{-1}\|^{n+2s} \|B\|^{2s}
   K(s,\hat\CO)^2 \Bigl( 3 + \frac {\CPFb(s,\hat\CO)^2}{s(1-s)} \Bigr)
   |v|_{L^2(\CO),H^1(\CO),s}^2
\]
for any $v\in H^s(\CO)$ and $s\in (0,1)$
with $K(s,\hat\CO)$ from Proposition~\ref{prop_norms} and
     $\CPFb(s,\hat\CO)$ from Proposition~\ref{prop_PF2}.
\item[(ii)]
\[
   |v|_{L^2(\CO),H^1(\CO),s}^2
   \le
   |\det B|^{-1} \|B\|^{n+2s} \|B^{-1}\|^{2s}
   k(s,\hat\CO)^{-2} \Bigl( 1 + \CPFa(s,\hat\CO)^2 \Bigr)
   |v|_{s,\CO}^2
\]
for any $v\in H^s(\CO)$ and $s\in (0,1)$
with $k(s,\hat\CO)$ from Proposition~\ref{prop_norms} and
     $\CPFa(s,\hat\CO)$ from Proposition~\ref{prop_PF1}.
\end{enumerate}
\end{theorem}

\begin{proof}
On a fixed domain $\hat\CO$ we obtain, by combining Lemmas~\ref{la1} and \ref{la3},
the equivalence of semi-norms:
\be \label{pf_thm1_1}
   |\hat v|_{s,\hat\CO}^2
   \le
   |\hat v|_{s,\hat\CO,\inf}^2
   \le
   K(s,\hat\CO)^2
   \Bigl( 3 + \frac {\CPFb(s,\hat\CO)^2}{s(1-s)} \Bigr)
   |\hat v|_{L^2(\hat\CO),H^1(\hat\CO),s}^2
\ee
and
\be \label{pf_thm1_2}
   |\hat v|_{L^2(\hat\CO),H^1(\hat\CO),s}^2
   \le
   k(s,\hat\CO)^{-2}
   |\hat v|_{s,\hat\CO,\inf}^2
   \le
   k(s,\hat\CO)^{-2}
   \Bigl(1+\CPFa(s,\hat\CO)^2\Bigr) |\hat v|_{s,\hat\CO}^2.
\ee
The first assertion of the theorem then follows by combining \eqref{pf_thm1_1} with the
transformation properties of the semi-norms by Lemma~\ref{la_scale_semi}:
\begin{align*}
   |v|_{s,\CO}^2
   &\le
   |\det B|^2 \|B^{-1}\|^{n+2s} |\hat v|_{s,\hat\CO}^2
   \\
   &\le
   |\det B|^2 \|B^{-1}\|^{n+2s}
   K(s,\hat\CO)^2
   \Bigl( 3 + \frac {\CPFb(s,\hat\CO)^2}{s(1-s)} \Bigr)
   |\hat v|_{L^2(\hat\CO),H^1(\hat\CO),s}^2
   \\
   &\le
   |\det B|\, \|B^{-1}\|^{n+2s} \|B\|^{2s}
   K(s,\hat\CO)^2
   \Bigl( 3 + \frac {\CPFb(s,\hat\CO)^2}{s(1-s)} \Bigr)
   |v|_{L^2(\CO),H^1(\CO),s}^2.
\end{align*}
The second assertion of the theorem is proved by a combination of \eqref{pf_thm1_2}
with the transformation properties by Lemma~\ref{la_scale_semi}.
\end{proof}

The next two theorems study the other pairs of semi-norms for equivalence
in combination with affine maps, $(|\cdot|_{s,\CO}, |\cdot|_{s,\CO,\inf})$ and
$(|\cdot|_{L^2(\CO),H^1(\CO),s}, |\cdot|_{s,\CO,\inf})$.

\begin{theorem} \label{thm2}
Let $\hat\CO\subset\R^n$ be a bounded, connected Lipschitz domain and let $\CO$ be the affinely
transformed domain defined by \eqref{trafo}. Then there hold the following relations.
\begin{enumerate}
\item[(i)]
\[
   |v|_{s,\CO}
   \le
   |v|_{s,\CO,\inf}
   \qquad\forall v\in H^s(\CO),\ \forall s\in (0,1),
\]
\item[(ii)]
\[
   |v|_{s,\CO,\inf}^2
   \le
   \Bigl( 1 + |\det B|^{-1} \|B\|^{n+2s} \CPFa(s,\hat\CO)^2 \Bigr)
   |v|_{s,\CO}^2
   \qquad\forall v\in H^s(\CO),\ \forall s\in (0,1)
\]
with $\CPFa(s,\hat\CO)$ being the number from Proposition~\ref{prop_PF1}.
\end{enumerate}
\end{theorem}

\begin{proof}
Assertion (i) is a repetition of the first estimate in Lemma~\ref{la1}.

To show the second assertion we use Proposition~\ref{prop_PF1} and
Lemma~\ref{la_scale_semi} to deduce that
\begin{align*}
   \inf_{c\in\fR} \|v+c\|_{0,\CO}^2
   &=
   |\det B|\, \inf_{c\in\fR} \|\hat v+c\|_{0,\hat \CO}^2
   \le
   |\det B|\, \CPFa(s,\hat\CO)^2 |\hat v|_{s,\hat\CO}^2
   \\
   &\le
   |\det B|^{-1} \|B\|^{n+2s} \CPFa(s,\hat\CO)^2 |v|_{s,\CO}^2.
\end{align*}
The assertion then follows by the definition of the semi-norm $|\cdot|_{s,\CO,\inf}$.
\end{proof}

\begin{theorem} \label{thm3}
Let $\hat\CO\subset\R^n$ be a bounded, connected Lipschitz domain and let $\CO$
be the affinely transformed domain defined by \eqref{trafo}.
Then there hold the following relations.
\begin{enumerate}
\item[(i)]
\[
   |v|_{L^2(\CO),H^1(\CO),s}^2
   \le
   \|B^{-1}\|^{2s} \max\{|\det B|^{-1} \|B\|^{n+2s}, 1\}\;
   k(s,\hat\CO)^{-2}
   |v|_{s,\CO,\inf}^2
\]
for any $v\in H^s(\CO)$ and $s\in (0,1)$
with $k(s,\hat\CO)$ from Proposition~\ref{prop_norms},
\item[(ii)]
\[
   |v|_{s,\CO,\inf}^2
   \le
   \|B\|^{2s} \max\{|\det B|\, \|B^{-1}\|^{n+2s}, 1\}\;
   K(s,\hat\CO)^2 \Bigl( 3 + \frac {\CPFb(s,\hat\CO)^2}{s(1-s)} \Bigr)
   |v|_{L^2(\CO),H^1(\CO),s}^2
\]
for any $v\in H^s(\CO)$ and $s\in (0,1)$
with $K(s,\hat\CO)$ from Proposition~\ref{prop_norms} and
     $\CPFb(s,\hat\CO)$ from Proposition~\ref{prop_PF2}.
\end{enumerate}
\end{theorem}

\begin{proof}
By Lemmas~\ref{la_scale_semi}, \ref{la3}, and \ref{la_scale_semi3} we obtain
\begin{align*}
\lefteqn{
   |v|_{L^2(\CO),H^1(\CO),s}^2
   \le
   |\det B|\,\|B^{-1}\|^{2s}
   |\hat v|_{L^2(\hat\CO),H^1(\hat\CO),s}^2
   \le
   |\det B|\,\|B^{-1}\|^{2s}
   k(s,\hat\CO)^{-2} |\hat v|_{s,\hat\CO,\inf}^2
}
   \\
   &\le
   |\det B|\,\|B^{-1}\|^{2s}
   k(s,\hat\CO)^{-2}
   \Bigl(|\det B|^{-2}\|B\|^{n+2s} |v|_{s,\CO}^2 +
         |\det B|^{-1} \inf_{c\in\fR} \|v+c\|_{0,\CO}^2\Bigr)
   \\
   &\le
   \|B^{-1}\|^{2s} \max\{|\det B|^{-1} \|B\|^{n+2s}, 1\}\;
   k(s,\hat\CO)^{-2} |v|_{s,\CO,\inf}^2.
\end{align*}
This is the first assertion. The second one follows analogously by the same lemmas:
\begin{align*}
   |v|_{s,\CO,\inf}^2
   &\le
   |\det B|^2 \|B^{-1}\|^{n+2s} |\hat v|_{s,\hat\CO}^2
   +
   |\det B|\, \inf_{c\in\fR} \|\hat v+c\|_{0,\hat\CO}^2
   \\
   &\le
   |\det B|\, \max\{|\det B|\, \|B^{-1}\|^{n+2s}, 1\}\;
   |\hat v|_{s,\hat\CO,\inf}^2
   \\
   &\le
   |\det B|\, \max\{|\det B|\, \|B^{-1}\|^{n+2s}, 1\}\;
   K(s,\hat\CO)^2 \Bigl( 3 + \frac {\CPFb(s,\hat\CO)^2}{s(1-s)} \Bigr)
   |\hat v|_{L^2(\hat\CO),H^1(\hat\CO),s}^2
   \\
   &\le
   \max\{|\det B|\, \|B^{-1}\|^{n+2s}, 1\}\; \|B\|^{2s}
   K(s,\hat\CO)^2 \Bigl( 3 + \frac {\CPFb(s,\hat\CO)^2}{s(1-s)} \Bigr)
   |v|_{L^2(\CO),H^1(\CO),s}^2.
\end{align*}
\end{proof}

We end this section with establishing uniform equivalence of the semi-norms
$|\cdot|_{s,\CO}$ and $|\cdot|_{L^2(\CO),H^1(\CO),s}$ for shape-regular domains.
Three of the four remaining bounds for other combinations of semi-norms are
uniform under further restrictions on the diameter of the domain.

Let us introduce some notation. We consider a bounded, connected Lipschitz domain
$\hat\CO\subset\R^n$ and maps of $\hat\CO$ onto domains $\CO$ where
the ratio $\rho_{\CO}:=D_{\CO}/d_{\CO}$ is controlled.
Here, $D_{\CO}$ denotes the diameter of $\CO$ and $d_{\CO}$ is the supremum of
the diameters of all balls contained in $\CO$. In the case of finite elements
(or convex polygons) boundedness of $\rho$ is referred to as {\em shape regularity}
of $\CO$. Also, when defining $d_{\CO}$ with balls with respect to which $\CO$ is
star-shaped, then $\rho_{\CO}$ is referred to as {\em chunkiness parameter}.

Using the notation \eqref{trafo} there holds
\be \label{aff1}
   \|B\|             \le \frac{D_{\CO}}{d_{\hat\CO}}
                     =   \frac{D_{\CO}}{D_{\hat\CO}}\, \rho_{\hat\CO},\quad
   \|B^{-1}\|        \le \frac{D_{\hat\CO}}{d_{\CO}}
                     =   \frac{D_{\hat\CO}}{D_{\CO}}\, \rho_{\CO},\quad
   \|B\|\,\|B^{-1}\| \le \rho_{\CO}\,\rho_{\hat\CO},
\ee
cf., e.g., \cite{Braess_97_FET}. Furthermore, we conclude that
\be \label{aff2}
   |\det B| = \frac {|\CO|}{|\hat\CO|} \le \frac {D_{\CO}^n}{d_{\hat\CO}^n},\quad
   |\det B|^{-1} \le \frac {D_{\hat\CO}^n}{d_{\CO}^n}
                 = \rho_{\CO}^n \frac {D_{\hat\CO}^n}{D_{\CO}^n}.
\ee
With this notation, the results of Theorems~\ref{thm1}--\ref{thm3} imply the following.

\begin{theorem} \label{thm_main}
Let $\CO$ be the affine map of a bounded connected Lipschitz domain $\hat\CO\subset\R^n$,
cf. \eqref{trafo}.

\begin{enumerate}
\item[(i)]
The semi-norms $|\cdot|_{s,\CO}$ and $|\cdot|_{L^2(\CO),H^1(\CO),s}$ are uniformly equivalent
for a family of shape-regular domains $\CO$:
\begin{align*}
   |v|_{s,\CO}^2
   &\le
   \rho_{\CO}^{n+2s} \rho_{\hat\CO}^{n+2s}
   K(s,\hat\CO)^2 \Bigl( 3 + \frac {\CPFb(s,\hat\CO)^2}{s(1-s)} \Bigr)
   |v|_{L^2(\CO),H^1(\CO),s}^2,
   \\
   |v|_{L^2(\CO),H^1(\CO),s}^2
   &\le
   \rho_{\CO}^{n+2s} \rho_{\hat\CO}^{n+2s}
   k(s,\hat\CO)^{-2} \Bigl( 1 + \CPFa(s,\hat\CO)^2 \Bigr)
   |v|_{s,\CO}^2
\end{align*}
for any $v\in H^s(\CO)$ and $s\in (0,1)$.
Here, $k(s,\hat\CO)$, $K(s,\hat\CO)$ are the numbers from Proposition~\ref{prop_norms} and
$\CPFa(s,\hat\CO)$, $\CPFb(s,\hat\CO)$ are as in Propositions~\ref{prop_PF1}, \ref{prop_PF2},
respectively.

\item[(ii)]
The semi-norms $|\cdot|_{s,\CO}$ and $|\cdot|_{s,\CO,\inf}$ are uniformly equivalent
for a family of uniformly bounded, shape-regular domains $\CO$:
\begin{align*}
   |v|_{s,\CO}
   &\le
   |v|_{s,\CO,\inf},
   \\
   |v|_{s,\CO,\inf}^2
   &\le
   \Bigl( 1 + \frac {D_{\CO}^{2s}}{D_{\hat\CO}^{2s}} \rho_{\CO}^n \rho_{\hat\CO}^{n+2s}
              \CPFa(s,\hat\CO)^2 \Bigr)
   |v|_{s,\CO}^2
\end{align*}
for any $v\in H^s(\CO)$ and $s\in (0,1)$.
Here, $\CPFa(s,\hat\CO)$ is the number from Proposition~\ref{prop_PF1}.

\item[(iii)]
\begin{enumerate}
\item[a)]
For a family of shape-regular domains $\CO$ whose diameters are bounded from below by a positive
constant, the semi-norm $|\cdot|_{L^2(\CO),H^1(\CO),s}$ is uniformly bounded by
$|\cdot|_{s,\CO,\inf}$:
\begin{align*}
   |v|_{L^2(\CO),H^1(\CO),s}^2
   &\le
   \max\{\rho_{\CO}^n \rho_{\hat\CO}^{n-2s}, D_{\CO}^{-2s} D_{\hat\CO}^{2s}\}\;
   \rho_{\CO}^{2s}
   k(s,\hat\CO)^{-2}
   |v|_{s,\CO,\inf}^2
\end{align*}
for any $v\in H^s(\CO)$ and $s\in (0,1)$.
\item[b)]
For a family of uniformly bounded, shape-regular domains $\CO$,
the semi-norm $|\cdot|_{s,\CO,\inf}$ is uniformly bounded by $|\cdot|_{L^2(\CO),H^1(\CO),s}$:
\[
\begin{split}
   |v|_{s,\CO,\inf}^2
   \le
   \max\{\rho_{\CO}^{n+2s} \rho_{\hat\CO}^n, D_{\CO}^{2s} D_{\hat\CO}^{-2s}\}\;
   \rho_{\hat\CO}^{2s}
\hspace{8em}
   \\
   K(s,\hat\CO)^2 \Bigl( 3 + \frac {\CPFb(s,\hat\CO)^2}{s(1-s)} \Bigr)
   |v|_{L^2(\CO),H^1(\CO),s}^2
\end{split}
\]
for any $v\in H^s(\CO)$ and $s\in (0,1)$.
\end{enumerate}
Here, $k(s,\hat\CO)$, $K(s,\hat\CO)$ are the parameters from Proposition~\ref{prop_norms},
and $\CPFb(s,\hat\CO)$ is the number from Proposition~\ref{prop_PF2}.
\end{enumerate}
\end{theorem}

\begin{proof}
The assertions (i)--(iii) are a combination of Theorems~\ref{thm1}--\ref{thm3},
respectively, with the bounds provided by \eqref{aff1}, \eqref{aff2}.
\end{proof}

The uniform equivalence of the semi-norms $|\cdot|_{s,\CO}$ and
$|\cdot|_{L^2(\CO),H^1(\CO),s}$ for shape-regular domains is based on what one calls
their {\em scaling property}. It means that both semi-norms for functions on a domain
$\CO$ are uniformly equivalent to the respective semi-norm of the affinely
transformed functions onto a fixed domain $\hat\CO$, when one of the semi-norms is
multiplied by an appropriate number (it is a power of the diameter of $\CO$).
This property applies also to the norms $\|\cdot\|_{L^2(\CO),H^1_0(\CO),s}$
and $\|\cdot\|_{\sim,s,\CO}$, cf.~Lemma~\ref{la_scale}. Scaling properties are relevant
for the error analysis of piecewise polynomial approximations. We formulate the
result as a corollary to Lemmas~\ref{la_scale} and \ref{la_scale_semi}.

\begin{corollary} \label{cor_scale}
The norms $\|\cdot\|_{L^2(\CO),H^1_0(\CO),s}$, $\|\cdot\|_{\sim,s,\CO}$
and semi-norms $|\cdot|_{s,\CO}$, $|\cdot|_{L^2(\CO),H^1(\CO),s}$ are scalable
of order $D_{\CO}^{n-2s}$:
\[
\begin{split}
   D_{\CO}^{n-2s} \rho_{\CO}^{-n} D_{\hat\CO}^{2s-n} \rho_{\hat\CO}^{-2s}
   \|\hat v\|_{L^2(\hat \CO),H^1_0(\hat \CO),s}^2
   \le
   \|v\|_{L^2(\CO),H^1_0(\CO),s}^2
\hspace{6em}
   \\
   \le
   D_{\CO}^{n-2s} \rho_{\CO}^{2s} D_{\hat\CO}^{2s-n} \rho_{\hat\CO}^{n}
   \|\hat v\|_{L^2(\hat \CO),H^1_0(\hat \CO),s}^2,
   \\[1em]
   D_{\CO}^{n-2s} \rho_{\CO}^{-n} D_{\hat\CO}^{2s-n} \rho_{\hat\CO}^{-2s}
   \min\{\rho_{\CO}^{-n} \rho_{\hat\CO}^{-n},1\}\;
   \|\hat v\|_{\sim,s,\hat\CO}^2
   \le
   \|v\|_{\sim,s,\CO}^2
\hspace{6em}
   \\
   \le
   D_{\CO}^{n-2s} \rho_{\CO}^{2s} D_{\hat\CO}^{2s-n} \rho_{\hat\CO}^n
   \max\{\rho_{\CO}^n \rho_{\hat\CO}^n,1\}\;
   \|\hat v\|_{\sim,s,\hat\CO}^2
\end{split}
\]
for any $v\in \tilde H^s(\CO)$ and $s\in (0,1)$, and
\[
\begin{split}
   D_{\CO}^{n-2s} \rho_{\CO}^{-n} D_{\hat\CO}^{2s-n} \rho_{\hat\CO}^{-2s}
   |\hat v|_{L^2(\hat \CO),H^1(\hat \CO),s}^2
   \le
   |v|_{L^2(\CO),H^1(\CO),s}^2
\hspace{9em}
   \\
   \le
   D_{\CO}^{n-2s} \rho_{\CO}^{2s} D_{\hat\CO}^{2s-n} \rho_{\hat\CO}^{n}
   |\hat v|_{L^2(\hat \CO),H^1(\hat \CO),s}^2,
   \\[1em]
   D_{\CO}^{n-2s} \rho_{\CO}^{-2n} D_{\hat\CO}^{2s-n} \rho_{\hat\CO}^{-n-2s}
   |\hat v|_{s,\hat\CO}^2
   \le
   |v|_{s,\CO}^2
\hspace{9em}
   \\
   \le
   D_{\CO}^{n-2s} \rho_{\CO}^{n+2s} D_{\hat\CO}^{2s-n} \rho_{\hat\CO}^{2n}
   |\hat v|_{s,\hat\CO}^2
\end{split}
\]
for any $v\in H^s(\CO)$ and $s\in (0,1)$.
\end{corollary}
 
\begin{proof}
The bounds are a combination of Lemmas~\ref{la_scale} and \ref{la_scale_semi}
with \eqref{aff1}, \eqref{aff2}.
\end{proof}

\begin{remark}
The estimate by Theorem~\ref{thm_main} (iii) a) breaks down
when $D_{\CO}\to 0$. In fact, for a family of scaled domains $\CO_h$ with $D_{\CO_h}=h$
and a non-constant function $v$ scaled to a family $\{v_h\}$ of functions on $\{\CO_h\}$,
$|v_h|_{L^2(\CO_h),H^1(\CO_h),s}^2 \simeq h^{n-2s}$ by Corollary~\ref{cor_scale}
whereas $|v_h|_{s,\CO_h,\inf}^2\ge \inf_{c\in\R}\|v_h-c\|_{0,\CO_h}^2\simeq h^n$.
Therefore, the dependence on $D_{\CO}$ like $D_{\CO}^{-2s}$ of the upper bound
in Theorem~\ref{thm_main} (iii) a) is optimal.
\end{remark}

%%%%%%%%%%%%%%%%%%%%%%%%%%%%%%%%%%%%%%%%%%%%%%%%%%%%%%%%%%%%%%%%%%%%%%%%%%%%%%%%
\noindent
{\bf Acknowledgment.} We thank an anonymous referee for constructive remarks.

%\clearpage
%*flatex input: [paper.bbl]

% flatex input end: [paper.bbl]
%FLATEX-REM:\bibliographystyle{siam}
%FLATEX-REM:\bibliography{/home/norbert/tex/bib/bib,/home/norbert/tex/bib/heuer,/home/norbert/tex/bib/fem}

\end{document}